\documentclass[11pt]{amsart}

\usepackage{amsmath,amssymb,latexsym,amsthm,newlfont}
\usepackage{etoolbox,changepage}
\usepackage{thmtools,enumerate}
\usepackage{mathtools}

\usepackage{hyperref}
\usepackage{color} 
\usepackage{tikz-cd}
\usepackage[all]{xy}
\usepackage{pstricks}

\theoremstyle{plain}

\newtheorem{thm}{Theorem}[section]

\newtheorem{cor}[thm]{Corollary}

\newtheorem*{NonCon}{Nonexistence Conjecture}

\newtheorem{thmA}{Theorem}

\newtheorem{corA}[thmA]{Corollary}

\theoremstyle{definition}

\newtheorem{rem}[thm]{Remark}

\newcommand{\N}{\mathbb{N}}
\newcommand{\C}{\mathbb{C}}
\newcommand{\R}{\mathbb{R}}
\newcommand{\Q}{\mathbb{Q}}

\DeclareMathOperator{\Div}{Div}
\DeclareMathOperator{\NEb}{\overline{\mathrm{NE}}}
\DeclareMathOperator{\Supp}{Supp}
\DeclareMathOperator{\mult}{mult}

\begin{document}
	\title[Abundance for non-RC uniruled pairs]{Abundance for uniruled pairs\\ which are not rationally connected}
	
	\author{Vladimir Lazi\'c}
	\address{Fachrichtung Mathematik, Campus, Geb\"aude E2.4, Universit\"at des Saarlandes, 66123 Saarbr\"ucken, Germany}
	\email{lazic@math.uni-sb.de}
	
	\thanks{
		I was supported by the DFG-Emmy-Noether-Nachwuchsgruppe ``Gute Strukturen in der h\"oherdimensionalen birationalen Geometrie". I thank F.\ Meng, N.\ Tsakanikas and the referee for useful comments and suggestions.
	}
	
\begin{abstract}
One of the central aims of the Minimal Model Program is to show that a projective log canonical pair $(X,\Delta)$ with $K_X+\Delta$ pseudoeffective has a good model, i.e.\ a minimal model $(Y,\Delta_Y)$ such that $K_Y+\Delta_Y$ is semiample. The goal of this paper is to show that this holds if $X$ is uniruled but not rationally connected, assuming the Minimal Model Program in dimension $\dim X-1$. Moreover, if $X$ is rationally connected, then we show that the existence of a good minimal model for $(X,\Delta)$ follows from a nonexistence conjecture for a very specific class of rationally connected pairs of Calabi--Yau type.
\end{abstract}
	
	\maketitle
	\setcounter{tocdepth}{1}
	\tableofcontents
	
\section{Introduction}

One of the central aims of the Minimal Model Program (MMP) over the complex numbers is to show that a projective log canonical pair $(X,\Delta)$ with $K_X+\Delta$ pseudoeffective has a minimal model $(Y,\Delta_Y)$ on which the Abundance conjecture holds, i.e.\ such that some multiple of $K_Y+\Delta_Y$ is basepoint free. Such a minimal model is called a \emph{good model}. The Abundance conjecture and the existence of good models for log canonical pairs are among the most important open problems in higher dimensional birational geometry: indeed, they -- together with other results of the MMP -- imply that every projective variety with mild singularities can be built (up to birational equivalence) out of varieties whose canonical class is either ample, numerically trivial or anti-ample: in other words, whose curvature has the constant sign. This programme was completed in dimensions up to $3$ by work of many people; the proof of the Abundance conjecture was finalised in \cite{KMM94}. In higher dimensions the programme was settled for varieties of log general type in \cite{BCHM} and in \cite{CL12a,CL13} by different methods; in general, there has been a lot of progress, but the Abundance conjecture remains open.

The goal of this paper is to essentially solve the Abundance conjecture and the existence of good models for the class of pairs whose underlying varieties are uniruled but not rationally connected. Here, the word \emph{essentially} means that the problems are solved modulo the MMP in lower dimensions.

It is known at least since \cite{KMM94} that the behaviour of projective log canonical pairs $(X,\Delta)$ depends to a large extent on whether the underlying variety $X$ is covered by rational curves or not. In some sense, pairs whose underlying variety $X$ is uniruled are easier to work with due to the fact that if $X$ is additionally smooth, then $K_X$ is not pseudoeffective. For such pairs, several of the most important problems in the MMP were recently essentially solved: the Nonvanishing conjecture in \cite{DHP13,LM21} and the existence of minimal models in \cite{LT22}.

With notation as above, if $K_X+\Delta$ is pseudoeffective, $(X,\Delta)$ is klt and log smooth and $\Delta$ is a $\Q$-divisor, it has been known since \cite{DL15} that there exists some $0<\tau\leq1$ such that the pair $(X,\tau\Delta)$ has a good model, assuming the MMP in lower dimensions. However, it is not clear how to deduce the existence of a good model of $(X,\Delta)$ from the existence of a good model of $(X,\tau\Delta)$.

The following is the main result of the paper.

\begin{thmA}\label{thm:MRC}
Assume the existence of good models for non-uniruled klt pairs with boundaries with rational coefficients in dimension $n-1$.

Let $(X,\Delta)$ be a projective log canonical pair of dimension $n$ such that $X$ is uniruled but not rationally connected. If $K_X+\Delta$ is pseudoeffective, then $(X,\Delta)$ has a good model. In particular, if $K_X+\Delta$ is nef, then it is semiample.
\end{thmA}

Recall that the main result of \cite{DL15} was to reduce the problem of the existence of good models for uniruled klt pairs to that of non-uniruled klt pairs, under the same assumption in lower dimensions as in Theorem \ref{thm:MRC}. Thus, the previous result improves dramatically on \cite[Theorems 1.1 and 1.2]{DL15} when the underlying variety is not rationally connected, even in the klt category.

Since all the conjectures of the Minimal Model Program hold in dimension $3$, an immediate corollary is:

\begin{corA}\label{cor:dim4}
Let $(X,\Delta)$ be a projective log canonical pair of dimension $4$ such that $X$ is uniruled but not rationally connected. If $K_X+\Delta$ is pseudoeffective, then $(X,\Delta)$ has a good model. In particular, if $K_X+\Delta$ is nef, then it is semiample.
\end{corA}

In the situation as in Theorem \ref{thm:MRC}, one usually tries to run an MMP in order to obtain a certain Mori fibre space; indeed, this was the strategy in \cite{DHP13,DL15,LT22}. In this paper, I use a different approach. Even though the idea of the proof of Theorem \ref{thm:MRC} is relatively simple and the actual proof is quite short, it uses machinery which has only been obtained very recently. 

The starting idea of this paper is to consider the MRC fibration $\pi\colon X\dashrightarrow Z$, when the divisor $\Delta$ has rational coefficients and the pair $(X,\Delta)$ is klt; the definition and properties of MRC fibrations are presented in Section \ref{section:prelim}. Since $X$ is not rationally connected, we have $\dim Z>0$. One easily reduces to the case when $\pi$ is a morphism, and we have that the Kodaira dimension $\kappa(X,K_X+\Delta)$ is nonnegative by the main result of \cite{LM21}. Then one analyses $\kappa(F,K_F+\Delta|_F)$, where $F$ is a very general fibre of $\pi$; we have $\kappa(F,K_F+\Delta|_F)\geq0$ by induction. If $\kappa(F,K_F+\Delta|_F)<\dim F$, then one runs a relative $(K_X+\Delta)$-MMP over $Z$ to obtain a situation where one can apply the canonical bundle formula from \cite{Amb05a} and use the induction on the dimension to conclude; this idea has already been exploited in \cite{LP18a,LP20a}. Otherwise, one uses the subadditivity of the Kodaira dimension from \cite{KP17} to show that $\kappa(X,K_X+\Delta)>0$, which then allows to conclude essentially by the main result of \cite{Lai11}. 

In general, when $\Delta$ has real coefficients and $(X,\Delta)$ is log canonical, one has to work somewhat more and additionally employ results from \cite{DL15,LT22,LM21,HH20} to conclude.

In fact, Theorem \ref{thm:MRC} is a special case of the following more general result.

\begin{thmA}\label{thm:main}
Assume the existence of good models for non-uniruled klt pairs with boundaries with rational coefficients in dimension $n-1$.

Let $(X,\Delta)$ be a projective log canonical pair of dimension $n$ and let $f\colon X\dashrightarrow Y$ be a dominant rational map to a normal projective variety $Y$ such that $0<\dim Y<\dim X$ and $Y$ is not uniruled. If $K_X+\Delta$ is pseudoeffective, then $(X,\Delta)$ has a good model.
\end{thmA}

The proof follows a similar strategy as above. In fact, when $\Delta$ is a $\Q$-divisor, then the result follows almost immediately from the main technical result of \cite{Has20}, and the argument is essentially a footnote to \emph{op.\ cit.}: indeed, in the case of divisors with rational coefficients, the main contribution of this paper is to observe that \emph{op.\ cit.}\ applies to rational maps for which the base is not uniruled, such as MRC fibrations. As in the proof of Theorem \ref{thm:MRC}, the case when $\Delta$ is not a $\Q$-divisor is somewhat more involved. I could have written the proof of Theorem \ref{thm:MRC} by using \cite{Has20} instead of \cite{KP17}, which would have shortened the argument (note however, that \cite{Has20} uses \cite{KP17} as a starting point). I think the proof of Theorem \ref{thm:MRC} below makes the argument more transparent and its logic clearer, although this is a matter of taste.

Previously (apart from the case of semipositive canonical bundles \cite{LP18a,GM17} or large Euler--Poincar\'e characteristic \cite{LP18a}), the Abundance conjecture has been (essentially) solved for a pair $(X,\Delta)$ only when one knows the existence of a nontrivial map from $X$ to some abelian variety \cite{Fuj13,BirChen15,Hu16}. Theorem \ref{thm:main} extends those results, see Corollary \ref{cor:abelian}.

\medskip

I end the paper with a section on rationally connected pairs. Consider the following:

\begin{NonCon}\label{con:RC}
There does not exist a klt pair $(X,\Delta)$ such that $X$ is rationally connected, $\Delta$ is a nef $\Q$-divisor whose support is a prime divisor, $\kappa(X,\Delta)=0$, $K_X+\Delta\sim_\Q0$ and $\Delta\cdot C>0$ for every curve $C$ on $X$ passing through a very general point on $X$.
\end{NonCon}

The conjecture follows from the Abundance conjecture: indeed, if a pair $(X,\Delta)$ as in the conjecture existed, consider the klt pair $\big(X,(1+\varepsilon)\Delta\big)$ for some small positive rational number $\varepsilon$. Then $K_X+(1+\varepsilon)\Delta\sim_\Q\varepsilon\Delta$ is nef, hence semiample by the Abundance conjecture. The condition $\kappa(X,\Delta)=0$ implies $\Delta\sim_\Q0$, which contradicts the condition that $\Delta$ intersects non-trivially every curve through a very general point on $X$ (or, indeed, contradicts the assumption that the support of $\Delta$ is non-trivial).

The conjecture predicts the nonexistence of a very special class of pairs. Note that it is a priori irrelevant that $X$ is rationally connected; it is my hope that the presence of a lot of rational curves will make the eventual proof easier.

The main result of Section \ref{sec:4} is:

\begin{thmA}\label{thm:RC}
Assume the existence of good models for non-uniruled klt pairs with boundaries with rational coefficients in dimension $n-1$.

Let $(X,\Delta)$ be a projective log canonical pair of dimension $n$ such that $K_X+\Delta$ is pseudoeffective and $X$ is rationally connected. If the Nonexistence conjecture holds in dimension $n$, then $K_X+\Delta$ has a good model. 
\end{thmA}

I believe the proof of Theorem \ref{thm:RC} is interesting in its own right and is of independent interest.

\section{Preliminaries}\label{section:prelim}

I work over $\C$. Unless explicitly stated otherwise, all varieties in the paper are normal and projective. A \emph{fibration} is a projective surjective morphism with connected fibres. A \emph{birational contraction} is a birational map whose inverse does not contract any divisors. 

If $X$ is a smooth projective variety, $D$ is a pseudoeffective $\R$-divisor on $X$ and $\Gamma$ is a prime divisor on $X$, then $\sigma_\Gamma(D)$ denotes Nakayama's $\sigma$-function of $D$ along $\Gamma$, see\ \cite[Chapter III]{Nak04}.

The standard reference for the definitions and basic results on the singularities of pairs and the Minimal Model Program is \cite{KM98}. A \emph{pair} $(X,\Delta)$ in this paper always has a boundary $\Delta\geq0$. A pair $(X,\Delta)$ has a \emph{boundary with rational coefficients} if the coefficients of $\Delta$ are rational numbers and $K_X+\Delta$ is $\Q$-Cartier. Unless otherwise stated, in a pair $(X,\Delta)$ the boundary $\Delta$ always has real coefficients.

We recall the following fundamental result.

\begin{thm}\label{thm:fingen}
Let $f\colon X \to Y$ be a projective morphism of quasi-projective varieties and let $\Delta$ be an effective $\Q$-divisor on $X$ such that the pair $(X,\Delta)$ is klt. Then the relative canonical ring
$$R(X/Y,K_X + \Delta) = \bigoplus_{n\in\N}f_*\mathcal O_X\big(\lfloor n(K_X + \Delta)\rfloor\big)$$
is a finitely generated $\mathcal O_Y$-algebra.
\end{thm}

Theorem \ref{thm:fingen} is stated in this form in \cite[Theorem 6.6]{Kaw09}. The sketch of the proof goes as follows: when $\Delta$ is $f$-big, then this is \cite[Corollary 1.1.9]{BCHM} or \cite[Theorem A]{CL12a} (see the first paragraph on \cite[p.\ 2417]{CL12a} and \cite[Theorem 9]{CL13}). In the general case, by applying the canonical bundle formula \cite[Section 4]{FM00} (see also \cite[Sections 4 and 5]{FL20} for a thorough discussion) to the relative Iitaka fibration $g\colon X\dashrightarrow Z$ associated to $K_X+\Delta$ over $Y$, we obtain a klt pair $(Z,\Delta_Z)$ such that $K_Z+\Delta_Z$ if big over $Y$ and a Veronese subalgebra of $R(X/Y,K_X+\Delta)$ is isomorphic to a Veronese subalgebra of $R(Z/Y,K_Z+\Delta_Z)$, see \cite[Definition 2.24]{CL12a}. Then an easy trick allows to assume that $\Delta_Z$ is big over $Y$, see the proof of \cite[Theorem 1.2]{BCHM}, and we are done by above.

The following result of Shokurov and Birkar \cite[Proposition 3.2(3)]{Bir11} will be used several times in the arguments below.

\begin{thm}\label{thm:ShoBir}
Let $X$ be a $\Q$-factorial projective variety, let $D_1,\dots,D_r$ be prime divisors on $X$, and denote $V=\bigoplus_{i
=1}^r\R D_i\subseteq\Div_\R(X)$. Then:
\begin{enumerate}[\normalfont (a)]
\item the set $\mathcal N(V)=\{\Delta\in V\mid (X,\Delta)\text{ is log canonical and }K_X+\Delta\text{ is nef}\,\}$ is a rational polytope,
\item if $\Delta\in\mathcal N(V)$ such that $(X,\Delta)$ is klt, then there exist finitely many $\Q$-divisors $\Delta_i$ on $X$ and positive real numbers $r_i$ such that each pair $(X,\Delta_i)$ is klt, each $K_X+\Delta_i$ is nef, $\Delta=\sum r_i\Delta_i$ and $K_X+\Delta=\sum r_i(K_X+\Delta_i)$.
\end{enumerate}
\end{thm}

\begin{proof}
Part (a) follows immediately by applying \cite[Proposition 3.2(3)]{Bir11} to the collection of \emph{all} extremal rays of the cone of curves $\NEb(X)$. For (b), pick a rational polytope $\mathcal P\subseteq V$ which contains $\Delta$ such that $(X,\Delta')$ is klt for every $\Delta'\in\mathcal P$. Then $\mathcal P\cap\mathcal N(V)$ is a rational polytope by (a), and set $\Delta_i$ to be its vertices.
\end{proof}

\subsection*{Invariant and numerical Kodaira dimensions}

If $X$ is a normal projective variety and $D$ is a pseudoeffective $\R$-Cartier $\R$-divisor on $X$, then $\kappa_\iota(X,D)$ denotes the \emph{invariant Kodaira dimension} of $D$, see \cite{Cho08}; if the divisor $D$ has rational coefficients or $D\geq0$, its Kodaira dimension is denoted by $\kappa(X,D)$. I denote by $ \nu(X,D) $ the \emph{numerical dimension} of $ D $, see \cite[Chapter V]{Nak04}, \cite{Kaw85}; this was denoted by $\kappa_\sigma$ in \cite{Nak04}. 

I use frequently and without explicit mention the following properties:
\begin{enumerate}[\normalfont (a)]
\item if $D$ is an $\R$-Cartier $\R$-divisor on a normal projective variety $X$ and if $f\colon Y\to X$ is a surjective morphism from a normal projective variety $Y$, then
$$\kappa_\iota(X,D)=\kappa_\iota(Y,f^*D)\quad\text{and}\quad\nu(X,D)=\nu(Y,f^*D),$$
and if additionally $f$ is birational and $E$ is an effective $f$-exceptional divisor on $Y$, then
$$\kappa_\iota(X,D)=\kappa_\iota(Y,f^*D+E)\quad\text{and}\quad\nu(X,D)=\nu(Y,f^*D+E);$$
see for instance \cite[\S2.2]{LP18a} for references and discussion,
\item if $D_1$ and $D_2$ are effective $\R$-Cartier $\R$-divisors on a normal projective variety $X$ such that $\Supp D_1=\Supp D_2$, then $\kappa_\iota(X,D_1)=\kappa_\iota(X,D_2)$ and $\nu(X,D_1)=\nu(X,D_2)$; the proof is easy and the same as that of \cite[Lemma 2.9]{DL15}.
\end{enumerate}

\subsection*{Good models}
Let $X$ and $Y$ be normal varieties, and let $D$ be an $\R$-Cartier $\R$-divisor on $X$. A birational contraction $f\colon X\dashrightarrow Y$ is a \emph{good model for $D$} if $f_*D$ is $\R$-Cartier and semiample, and if there exists a resolution of indeterminacies $(p,q)\colon W\to X\times Y$ of the map $f$ such that $p^*D=q^*f_*D+E$, where $E\geq0$ is a $q$-exceptional $\R$-divisor which contains the whole $q$-exceptional locus in its support. 

The following results will be used often in the remainder of the paper.

\begin{thm}\label{thm:Lai}
Assume the existence of good models for non-uniruled klt pairs with boundaries with rational coefficients in dimension $n-1$.

Let $(X,\Delta)$ be a projective log canonical pair of dimension $n$ such that $\Delta$ is a $\Q$-divisor. If $\kappa(X,K_X+\Delta)\geq1$, then $(X,\Delta)$ has a good model.
\end{thm}

\begin{proof}
By \cite[Theorem 1.3 and Lemmas 2.3 and 2.4]{LM21} we may assume the existence of good models for log canonical pairs in dimensions at most $n-1$. 

If $(X,\Delta)$ is klt, then the result follows by combining \cite[Propositions 2.4 and 2.5, and Theorem 4.4]{Lai11}; note that \cite[Theorem 4.4]{Lai11} is stated for a terminal variety $X$, but the proof generalises to the context of klt pairs by replacing \cite[Lemma 2.2]{Lai11} with \cite[Lemma 2.10]{HX13}. Alternatively, this is a special case of \cite[Theorem 2.12]{HX13} by Theorem \ref{thm:fingen}; see also \cite[Theorem 1.3]{Has18a}.

In general, the pair $(X,\Delta)$ has a minimal model by \cite[Theorem B]{LT22}, hence we may assume that $K_X+\Delta$ is nef. Then $\kappa(X,K_X+\Delta)=\nu(X,K_X+\Delta)$ by \cite[Proposition 3.1]{Fuk02}, and we conclude by \cite[Lemma 4.1]{LM21}.
\end{proof}

\begin{thm}\label{thm:DL}
Assume the existence of good models for non-uniruled klt pairs with boundaries with rational coefficients in dimension $n-1$.
\begin{enumerate}[\normalfont (a)]
\item Let $(X,\Delta)$ be a log canonical pair of dimension $n$ such that $X$ is uniruled and $K_X+\Delta$ is pseudoeffective. Then $\kappa_\iota(X,K_X+\Delta)\geq0$.
\item Let $(X,\Delta)$ be a klt pair of dimension $n$ such that $\Delta$ is a $\Q$-divisor. Let $G\neq0$ be an effective $\Q$-Cartier $\Q$-divisor such that $(X,\Delta+G)$ is klt and $K_X+\Delta+G$ is pseudoeffective. Assume that $K_X+\Delta+(1-\varepsilon)G$ is not pseudoeffective for any $\varepsilon>0$. Then there exists a good model of $(X,\Delta+G)$. 
\end{enumerate}
\end{thm}

\begin{proof}
Part (a) follows from \cite[Theorem 1.1]{LM21}. For (b), note first that we may assume the existence of good models for log canonical pairs in dimensions at most $n-1$ by \cite[Theorem 1.3 and Lemmas 2.3 and 2.4]{LM21}. Then the result follows from \cite[Theorem 3.3]{DL15}.
\end{proof}

\begin{rem}
Let $(X,\Delta)$ be a projective log canonical pair such that $K_X+\Delta$ is nef. If $(X,\Delta)$ has a good model, then $K_X+\Delta$ is semiample; this follows from the proof of \cite[Lemma 4.1]{LM21}. I use this fact in the remainder of the paper without explicit mention.
\end{rem}

\subsection*{MRC fibrations}

A proper variety $X$ is:
\begin{enumerate}[\normalfont (a)]
\item \emph{uniruled} if a general point on $X$ lies in a rational curve in $X$,
 
\item \emph{rationally chain connected} if any two very general points on $X$ can be joined by a chain of rational curves, or equivalently, if \emph{any two points on $X$} can be joined by a chain of rational curves, see \cite[Corollary IV.3.5]{Kol96};

\item \emph{rationally connected} if any two very general points on $X$ can be joined by a rational curve, or equivalently, if \emph{any two points on $X$} can be joined by a rational curve, see \cite[Theorem 2.1]{KMM92}.
\end{enumerate} 

If $X$ is a normal proper variety, then it admits a dominant almost holomorphic map $\pi\colon X \dashrightarrow Z$, called  \emph{maximal rationally chain connected} or \emph{MRCC fibration}, to a smooth variety such that the complete fibres of $\pi$ are rationally chain connected and for a very general fibre $F$ of $\pi$, any rational curve in $X$ which intersects $F$ lies in $F$, see \cite[Section IV.5]{Kol96}.

Now, let $(X,\Delta)$ be a projective dlt pair. Then $X$ is rationally chain connected if and only if $X$ is rationally connected by \cite[Corollary 1.5(2)]{HM07a}. If $\pi\colon X \dashrightarrow Z$ is an MRCC fibration of $X$, this implies that the complete fibres of $\pi$ are rationally connected -- we say that $\pi$ is a \emph{maximal rationally connected fibration} or \emph{MRC fibration} of $X$. Moreover, the variety $Z$ is not uniruled by \cite[Corollary 1.4]{GHS03}, see \cite[Lemma 3.17]{LMPTX22} for details. This implies that $\dim Z=\dim X$ if and only if $X$ is not uniruled, and note that $Z$ is a point if and only if $X$ is rationally connected. 

\section{Proofs of the main results}

In this section I prove Theorems \ref{thm:MRC} and \ref{thm:main}; Corollary \ref{cor:abelian} is an immediate consequence.

I start with the following result which was essentially proved in \cite{LT22,LM21}.

\begin{thm}\label{thm:tau}
Assume the existence of good models for non-uniruled klt pairs with boundaries with rational coefficients in dimension $n-1$.

Let $(X,\Delta)$ be a $\Q$-factorial dlt pair of dimension $n$ such that $K_X+\Delta$ is pseudoeffective. 
\begin{enumerate}[\normalfont (a)]
\item If $K_X+\Delta-\varepsilon\lfloor\Delta\rfloor$ is not pseudoeffective for any $\varepsilon>0$, then $(X,\Delta)$ has a good model.
\item Assume additionally that $(X,\Delta)$ is klt. If $K_X+(1-\varepsilon)\Delta$ is not pseudoeffective for any $\varepsilon>0$, then $(X,\Delta)$ has a good model.
\end{enumerate}
\end{thm}

\begin{proof}
By \cite[Theorem 1.3 and Lemmas 2.3 and 2.4]{LM21} we may assume the existence of good models for log canonical pairs in dimensions at most $n-1$. 

Part (a) follows from \cite[Lemma 4.1 and Proposition 4.2]{LM21}. 

For (b), if $\Delta$ is a $\Q$-divisor, this follows from Theorem \ref{thm:DL}(b). In general, I follow closely the proof of \cite[Theorem 3.1]{LT22}. Analogously as in Step 1 of that proof, we may assume the following:
	
	\medskip
	
	\emph{Assumption 1.}
	There exists a fibration $\xi\colon X\to Y$ to a normal projective variety $Y$ with $\dim Y<\dim X$ such that:
	\begin{enumerate}
		\item[(a$_1$)] $\nu(F, (K_X+\Delta)|_F)=0$ and $h^1(F, \mathcal{O}_F)=0$ for a very general fibre $F$ of $\xi$,
		\item[(b$_1$)] $K_X+(1-\varepsilon)\Delta$ is not $\xi$-pseudoeffective for any $\varepsilon>0$,
		\item[(c$_1$)] $(X, \Delta)$ is log smooth.
	\end{enumerate}

	If $\dim Y=0$, then $Y$ is a point and $\nu(X, K_X+\Delta)=0$. By (b$_1$) the divisor $K_X$ is not pseudoeffective, hence $X$ is uniruled by \cite[Corollary 0.3]{BDPP}. Therefore, $\kappa(X, K_X+\Delta)\geq0$ by \cite[Theorem 1.1]{LM21}, and thus $\kappa(X, K_X+\Delta)=\nu(X, K_X+\Delta)\geq0$.  We conclude by \cite[Lemma 4.1]{LM21}.
	
	Otherwise, as in Step 3 of the proof of \cite[Theorem 3.1]{LT22}, only by replacing \cite[Theorem 2.21]{LT22} by \cite[Theorems B, E and 2.22]{LT22} in that proof, we may assume the following:
	
	\medskip
	
	\emph{Assumption 2.}
	There exists a fibration $g\colon X\to T$ to a normal projective variety $T$ such that:
	\begin{enumerate}
	\item[(a$_2$)] $\dim T<\dim X$, and the numerical equivalence over $T$ coincides with the $\R$-linear equivalence over $T$,
	\item[(b$_2$)] $K_X+\Delta\equiv_T0$.
	\end{enumerate}

		By \cite[Theorem 0.2]{Amb05} and \cite[Theorem 3.1]{FG12} there exists an effective $\mathbb{R}$-divisor $\Delta_T$ on $T$ such that $(T, \Delta_T)$ is klt and $K_X+\Delta\sim_{\mathbb{R}}g^*(K_T+\Delta_T)$. By the assumption in lower dimensions, we have 
		$$\kappa_{\iota}(T,K_T+\Delta_T)=\nu(T,K_T+\Delta_T),$$
		and hence $\kappa_{\iota}(X,K_X+\Delta)=\nu(X,K_X+\Delta)$. We conclude by \cite[Lemma 4.1]{LM21}.
\end{proof}

\begin{proof}[Proof of Theorem \ref{thm:MRC}]
By \cite[Theorem 1.3 and Lemmas 2.3 and 2.4]{LM21} we may assume the existence of good models for log canonical pairs in dimensions at most $n-1$. By \cite[Lemma 4.1]{LM21} it suffices to show that
\begin{equation}\label{eq:5}
\kappa_{\iota}(X, K_X+\Delta)=\nu(X, K_X+\Delta).
\end{equation}

\medskip

\emph{Step 1.}
Let $f\colon X\dashrightarrow Y$ be an MRC fibration of $X$ with $Y$ smooth, see \cite[Section IV.5]{Kol96}. Since $X$ uniruled but not rationally connected, we have $0<\dim Y<\dim X$, and $K_Y$ is pseudoeffective by \cite[Corollary 1.4]{GHS03} and \cite[Corollary 0.3]{BDPP}. Thus,
\begin{equation}\label{eq:44}
\kappa(Y,K_Y)\geq0
\end{equation}
by the assumption in lower dimensions.

Let $(p,q)\colon X'\to X\times Y$ be a resolution of indeterminacies of $f$ which is at the same time a log resolution of the pair $(X,\Delta)$. We may write
$$K_{X'}+\Delta'\sim_\R p^*(K_X+\Delta)+E,$$
where $\Delta'$ and $E$ are effective $\R$-divisors without common components. Then it suffices to show that $\kappa_{\iota}(X', K_{X'}+\Delta')=\nu(X', K_{X'}+\Delta')$. Therefore, by replacing $(X,\Delta)$ by $(X',\Delta')$ and $f$ by $q$, we may assume that the pair $(X,\Delta)$ is log smooth and that $f$ is a fibration. 

The divisor $K_X$ is not pseudoeffective since $X$ possesses a free rational curve through a general point, see \cite[\S4.2]{Deb01}.

\medskip

\emph{Step 2.}
In this step I prove the theorem under the following assumption:

\medskip

\emph{Assumption 1.}
The divisor $\Delta$ is a $\Q$-divisor and $\lfloor\Delta\rfloor=0$.

\medskip

Then we have 
$$\kappa(X,K_X+\Delta)\geq0$$
by Theorem \ref{thm:DL}(a).

If $\kappa(X,K_X+\Delta)\geq1$, then the result follows from Theorem \ref{thm:Lai}.

Assume now that
\begin{equation}\label{eq:43}
\kappa(X,K_X+\Delta)=0.
\end{equation}
Let $F$ be a very general fibre of $f$ and note that $0<\dim F<\dim X$. Since $K_F+\Delta|_F$ is pseudoeffective, we have $\kappa(F,K_F+\Delta|_F)\geq0$ by the assumption in lower dimensions.

If $\kappa(F,K_F+\Delta|_F)=\dim F$, then $\kappa(X,K_X+\Delta)\geq\dim F$ by \eqref{eq:44} and by \cite[Theorem 9.9]{KP17} (for $M=K_Y$). But this contradicts \eqref{eq:43}.

Therefore, we may assume that
\begin{equation}\label{eq:1b}
0\leq\kappa(F,K_F+\Delta|_F)<\dim F.
\end{equation}
The pair $(F,\Delta|_F)$ has a good model by the assumption in lower dimensions. Therefore, by \cite[Theorem 2.12]{HX13} and by Theorem \ref{thm:fingen} there exists a relative good minimal model $X_{\min}$ of $(X,\Delta)$ over $Y$, and let $\theta\colon X\dashrightarrow X_{\min}$ be the corresponding birational contraction. Denote $\Delta_{\min}:=\theta_*\Delta$ and let $\tau\colon X_{\min}\to T$ be the relative Iitaka fibration over $Y$ associated to $K_{X_{\min}}+\Delta_{\min}$. Then $\dim T<\dim X$ since $\kappa(F,K_F+\Delta|_F)<\dim F$ by \eqref{eq:1b}. There exists a $\Q$-divisor $A$ on $T$ which is ample over $Y$ such that $K_{X_{\min}}+\Delta_{\min}\sim_\Q\tau^*A$. 
\[
\xymatrix{ 
X \ar@{-->}[rr]^\theta \ar[rd]_f & & X_{\min} \ar[dl] \ar[d]^\tau \\
& Y & T \ar[l]
}
\]
By \cite[Theorem 0.2]{Amb05a} there exists an effective $\Q$-divisor $\Delta_T$ on $T$ such that the pair $(T,\Delta_T)$ is klt and
$$K_{X_{\min}}+\Delta_{\min}\sim_\Q\tau^*(K_T+\Delta_T),$$
and in particular, $K_T+\Delta_T$ is pseudoeffective. By the assumption in lower dimensions, we have $\kappa(T,K_T+\Delta_T)=\nu(T,K_T+\Delta_T)$, and hence 
$$\kappa(X_{\min},K_{X_{\min}}+\Delta_{\min})=\nu(X_{\min},K_{X_{\min}}+\Delta_{\min}),$$
which gives \eqref{eq:5} as desired.

\medskip

\emph{Step 3.}
In this step I prove the theorem under the following assumption:

\medskip

\emph{Assumption 2.}
The divisor $\Delta$ is an $\R$-divisor and $\lfloor\Delta\rfloor=0$.

\medskip

The pair $(X,\Delta)$ has a minimal model $(Z,\Delta_Z)$ by \cite[Theorem C]{LT22}. By Theorem \ref{thm:ShoBir} there exist finitely many $\Q$-divisors $\Delta_i$ on $Z$ and positive real numbers $r_i$ such that each pair $(Z,\Delta_i)$ is klt, each $K_Z+\Delta_i$ is nef and $K_Z+\Delta_Z=\sum r_i(K_Z+\Delta_i)$. By Step 2 there exist semiample $\Q$-divisors $D_i$ such that $K_Z+\Delta_i\sim_\Q D_i$, hence the divisor $K_Z+\Delta_Z\sim_\R\sum r_iD_i$ is semiample.

\medskip

\emph{Step 4.}
Finally, it remains to consider the case $\lfloor\Delta\rfloor\neq0$. If $K_X+\Delta-\varepsilon\lfloor\Delta\rfloor$ is not pseudoeffective for any $\varepsilon>0$, then we conclude by Theorem \ref{thm:tau}(a). Otherwise, pick $\varepsilon>0$ such that $K_X+\Delta-\varepsilon\lfloor\Delta\rfloor$ is pseudoeffective. Then by Step 3 there exists an $\mathbb{R}$-divisor $D\geq0$ such that $K_X+\Delta-\varepsilon\lfloor\Delta\rfloor\sim_\R D$. Pick $0<\delta<\varepsilon$. Then
\begin{equation}\label{eq:3}
K_X+\Delta-\delta\lfloor\Delta\rfloor\sim_\R D+(\varepsilon-\delta)\lfloor\Delta\rfloor\quad\text{and}\quad K_X+\Delta\sim_\R D+\varepsilon\lfloor\Delta\rfloor.
\end{equation} 
Since $\big(X,\Delta-\delta\lfloor\Delta\rfloor\big)$ is a klt pair and $K_X+\Delta-\delta\lfloor\Delta\rfloor$ is pseudoeffective, the pair $\big(X,\Delta-\delta\lfloor\Delta\rfloor\big)$ has a good model by Step 3, and in particular,
\begin{equation}\label{eq:4}
\kappa_\iota(X,K_X+\Delta-\delta\lfloor\Delta\rfloor)=\nu(X,K_X+\Delta-\delta\lfloor\Delta\rfloor).
\end{equation}
Since $\Supp(D+(\varepsilon-\delta)\lfloor\Delta\rfloor)=\Supp(D+\varepsilon\lfloor\Delta\rfloor)$, we obtain \eqref{eq:5} from \eqref{eq:3} and \eqref{eq:4}. This concludes the proof.
\end{proof}

\begin{proof}[Proof of Theorem \ref{thm:main}]
By \cite[Theorem 1.3 and Lemmas 2.3 and 2.4]{LM21} we may assume the existence of good models for log canonical pairs in dimensions at most $n-1$. By \cite[Lemma 4.1]{LM21} it suffices to show that 
$$\kappa_{\iota}(X, K_X+\Delta)=\nu(X, K_X+\Delta).$$

\medskip

\emph{Step 1.}
As in the second paragraph of Step 1 of the proof of Theorem \ref{thm:MRC}, by additionally replacing $f$ by its Stein factorisation, we may assume that the pair $(X,\Delta)$ is log smooth and that $f$ is a fibration. 

Let $\widehat Y\to Y$ be a desingularisation of $Y$ and let $\pi\colon X\dashrightarrow \widehat Y$ be the resulting rational map. Let $(\alpha,\beta)\colon \widehat X\to X\times \widehat Y$ be a resolution of indeterminacies of $\pi$ which is at the same time a log resolution of the pair $(X,\Delta)$.
	\begin{center}
		\begin{tikzcd}
			\widehat X \arrow[d, "\alpha" swap] \arrow[r, "\beta"] & \widehat Y \arrow[d]  \\
			X \arrow[ru, dashed, "\pi"] \arrow[r, "f"] & Y 
		\end{tikzcd}
	\end{center} 
We may write
$$K_{\widehat X}+\widehat \Delta\sim_\R \alpha^*(K_X+\Delta)+G,$$
where $\widehat\Delta$ and $G$ are effective $\R$-divisors without common components. Then it suffices to show that $\kappa_{\iota}\big(\widehat X, K_{\widehat X}+\widehat \Delta\big)=\nu\big(\widehat X, K_{\widehat X}+\widehat \Delta\big)$. Therefore, by replacing $(X,\Delta)$ by $\big(\widehat X, \widehat \Delta\big)$, $Y$ by $\widehat Y$ and $f$ by $\beta$, we may assume that the pair $(X,\Delta)$ is log smooth, that $Y$ is smooth and that $f$ is a fibration.

The divisor $K_Y$ is then pseudoeffective by \cite[Corollary 0.3]{BDPP}, hence
\begin{equation}\label{eq:45}
\kappa(Y,K_Y)\geq0
\end{equation}
by the assumption in lower dimensions.

\medskip

\emph{Step 2.}
Assume in this step that $\Delta$ is a $\Q$-divisor.

Let $F$ be a very general fibre of $f$. Since $K_F+\Delta|_F$ is pseudoeffective, we have $\kappa(F,K_F+\Delta|_F)=\nu(F,K_F+\Delta|_F)$ by the assumption in lower dimensions.

Then $\kappa(X,K_X+\Delta)=\nu(X,K_X+\Delta)$ by \eqref{eq:45} and by \cite[Theorem 1.4(2)]{Has20} (for $M=K_Y$). 

\medskip

\emph{Step 3.}
In this step I assume that $\Delta$ is an $\R$-divisor and $\lfloor\Delta\rfloor=0$.

\medskip

If $K_X+(1-\varepsilon)\Delta$ is not pseudoeffective for all $\varepsilon>0$, then we conclude by Theorem \ref{thm:tau}(b).

Thus, we may assume that there exists $0<\varepsilon\ll1$ such that $K_X+(1-\varepsilon)\Delta$ is pseudoeffective. In particular, there exists a $\Q$-divisor $\widetilde{\Delta}$ such that 
$$(1-\varepsilon)\Delta\leq\widetilde\Delta\leq\Delta\quad\text{and}\quad K_X+\widetilde\Delta\text{ is pseudoeffective}.$$
By Step 2, we have $\kappa\big(X,K_X+\widetilde\Delta\big)\geq0$, so that $\kappa_\iota(X,K_X+\Delta)\geq0$. The pair $(X,\Delta)$ has a minimal model $(X',\Delta')$ by \cite[Theorem B]{LT22}, and observe that there exists a dominant rational map $X'\dashrightarrow Y$. By Theorem \ref{thm:ShoBir} there exist finitely many $\Q$-divisors $\Delta_i$ on $X'$ and positive real numbers $r_i$ such that each pair $(X',\Delta_i)$ is klt, each $K_{X'}+\Delta_i$ is nef and $K_{X'}+\Delta'=\sum r_i(K_{X'}+\Delta_i)$. By Step 2 there exist semiample $\Q$-divisors $D_i$ such that $K_{X'}+\Delta_i\sim_\Q D_i$, hence the divisor $K_{X'}+\Delta'\sim_\R\sum r_iD_i$ is semiample.

\medskip

\emph{Step 4.}
Finally, if $\lfloor\Delta\rfloor\neq0$, we conclude as in Step 4 of the proof of Theorem \ref{thm:MRC}.
\end{proof}

The following result complements \cite{Hu16}. Note that the extension theorem from \cite{DHP13} is not needed, see \cite[Remark 3.7]{Hu16}.

\begin{cor}\label{cor:abelian}
Assume the existence of good models for non-uniruled klt pairs with boundaries with rational coefficients in dimension $n-1$.

Let $(X,\Delta)$ be a projective log canonical pair of dimension $n$ such that $K_X+\Delta$ is pseudoeffective, and let $f\colon X\dashrightarrow A$ be a non-trivial rational map to an abelian variety $A$. Then $(X,\Delta)$ has a good model.
\end{cor}

\begin{proof}
As in Step 1 of the proof of Theorem \ref{thm:MRC}, we may assume that the pair $(X,\Delta)$ is log smooth and $f$ is a morphism. We may assume that $A$ is the Albanese variety of $X$. If $\dim f(X)<\dim X$, then we conclude by Theorem \ref{thm:main}, since the image of $f$ cannot be uniruled as $A$ contains no rational curves. 

Otherwise, $X$ is of maximal Albanese dimension. If $\Delta$ is a $\Q$-divisor and $(X,\Delta)$ is klt, then we conclude by \cite[Theorem 1.1]{Fuj13}. In general, we conclude as in Steps 3 and 4 of the proof of Theorem \ref{thm:main}.
\end{proof}

\section{On rationally connected pairs}\label{sec:4}

As announced in the introduction, in this section I reduce the problem of existence of good models for rationally connected log canonical pairs to a nonexistence statement for a very explicit class of rationally connected varieties of Calabi--Yau type.

\begin{proof}[Proof of Theorem \ref{thm:RC}]
By \cite[Theorem 1.3 and Lemmas 2.3 and 2.4]{LM21} we may assume the existence of good models for log canonical pairs in dimensions at most $n-1$. By \cite[Lemma 4.1]{LM21} it suffices to show that 
$$\kappa_{\iota}(X, K_X+\Delta)=\nu(X, K_X+\Delta).$$
By passing to a log resolution we may assume that $(X,\Delta)$ is log smooth. Then $K_X$ is not pseudoeffective since $X$ possesses a very free rational curve through a general point, see \cite[\S4.3]{Deb01}.

If the theorem holds for klt pairs with boundaries with rational coefficients, then it holds for all log canonical pairs as in Steps 3 and 4 of the proof of Theorem \ref{thm:MRC}.

\medskip

Therefore, from now on I assume that the pair $(X,\Delta)$ is log smooth, that $\lfloor\Delta\rfloor=0$ and that $\Delta$ is a $\Q$-divisor. We have $\kappa(X,K_X+\Delta)\geq0$ by Theorem \ref{thm:DL}(a), and by Theorem \ref{thm:Lai} we may assume that
\begin{equation}\label{eq:14}
\kappa(X,K_X+\Delta)=0.
\end{equation}

\medskip

\emph{Step 1.}
If $K_X+\tau\Delta$ is not pseudoeffective for any $\tau<1$, then we conclude by Theorem \ref{thm:tau}(b). 

Otherwise, pick a rational number $0<\tau<1$ such that $K_X+\tau\Delta$ is pseudoeffective. Then $\kappa(X,K_X+\tau\Delta)\geq0$ by Theorem \ref{thm:DL}(a), hence there exists a $\Q$-divisor $D_\tau\geq0$ such that $K_X+\tau\Delta\sim_\Q D_\tau$. Then for $D:=D_\tau+(1-\tau)\Delta\geq0$ we have
\begin{equation}\label{eq:7}
K_X+\Delta\sim_\Q D,
\end{equation}
and in particular, $\Supp\Delta\subseteq\Supp D$. We may replace $\Delta$ by $\Delta+\varepsilon D$ and $D$ by $(1+\varepsilon)D$ for some rational number $0<\varepsilon\ll1$, so we may assume that $\Supp\Delta=\Supp D$.

\medskip

\emph{Step 2.}
Let 
$$\Delta=\sum\delta_i D_i\quad\text{and}\quad D=\sum d_i D_i,$$
where $D_i$ are prime divisors on $X$, and $\delta_i$ and $d_i$ are positive rational numbers.

If $\delta_i\leq d_i$ for all $i$, then $K_X\sim_\Q\sum(d_i-\delta_i)D_i\geq0$, a contradiction since $K_X$ is not pseudoeffective. Therefore, we may assume without loss of generality that $\delta_1>d_1$.

Since
\begin{equation}\label{eq:6}
K_X+\Delta-d_1 D_1\sim_\Q \sum_{i\neq1}d_i D_i,
\end{equation}
the divisor $K_X+\Delta-d_1D_1$ is pseudoeffective, and I claim that
\begin{equation}\label{eq:claim}
K_X+\Delta-(d_1+\varepsilon)D_1\text{ is not pseudoeffective for any }\varepsilon>0.
\end{equation}
Indeed, assume that there exists a rational number $0<\varepsilon\ll1$ such that $\delta_1>d_1+\varepsilon$ and $K_X+\Delta-(d_1+\varepsilon)D_1$ is pseudoeffective. Note that 
$$\Delta-(d_1+\varepsilon)D_1=(\delta_1-d_1-\varepsilon)D_1+\sum_{i\neq1}\delta_i D_i.$$
Then by Theorem \ref{thm:DL}(a) there exists a $\Q$-divisor $D'\geq0$ such that 
$$K_X+\Delta-(d_1+\varepsilon)D_1\sim_\Q D',$$
and thus $D\sim_\Q D'+(d_1+\varepsilon)D_1$ by \eqref{eq:7}. Since $\kappa(X,D)=0$ by \eqref{eq:14} and \eqref{eq:7}, we have $D=D'+(d_1+\varepsilon)D_1$, and hence $\mult_{D_1}D\geq d_1+\varepsilon$, a contradiction which proves \eqref{eq:claim}.

\medskip

\emph{Step 3.}
Therefore, the pair $(X,\Delta-d_1 D_1)$ has a good model by \eqref{eq:claim} and by Theorem \ref{thm:DL}(b)\footnote{Apply Theorem \ref{thm:DL}(b) for $\Delta=\sum_{i\neq1}\delta_i D_i$ and $G=(\delta_1-d_1)D_1$.}. Since $\kappa(X,K_X+\Delta-d_1 D_1)\leq\kappa(X,K_X+\Delta)$, by \eqref{eq:14} we obtain 
$$\kappa(X,K_X+\Delta-d_1 D_1)=\nu(X,K_X+\Delta-d_1 D_1)=0.$$
By \eqref{eq:6} and by \cite[Corollary V.1.12]{Nak04} this yields
$$\sigma_{D_i}(K_X+\Delta-d_1 D_1)=d_i\quad\text{for all }i>1.$$ 
Pick a rational number $0<\mu\ll1$ such that
$$\sigma_{D_i}(K_X+\Delta-d_1 D_1+\mu D_1)>0\quad\text{for all }i>1;$$
this is possible by \cite[Lemma III.1.7(2)]{Nak04}. By \eqref{eq:6} we have
\begin{equation}\label{eq:8}
K_X+\Delta-d_1 D_1+\mu D_1\sim_\Q \mu D_1+\sum_{i\neq1}d_i D_i,
\end{equation}
and since $\Supp D=\Supp(\mu D_1+\sum_{i\neq1}d_i D_i)$, from \eqref{eq:7} and \eqref{eq:8} we obtain
$$\kappa(X,K_X+\Delta)=\kappa(X,K_X+\Delta-d_1 D_1+\mu D_1)$$
and
$$\nu(X,K_X+\Delta)=\nu(X,K_X+\Delta-d_1 D_1+\mu D_1).$$
Therefore, by replacing $(X,\Delta)$ by $(X,\Delta-d_1 D_1+\mu D_1)$, we may additionally assume that
\begin{equation}\label{eq:1a}
\sigma_{D_i}(K_X+\Delta)>0\quad\text{for all }i>1.
\end{equation}

\medskip

\emph{Step 4.}
We run a $(K_X+\Delta)$-MMP with scaling of an ample divisor. This MMP terminates by \eqref{eq:14} and by \cite[Theorem F]{LT22}, and all divisors $D_i$ for $i>1$ are contracted by this MMP by \eqref{eq:1a} and by \cite[Th\'eor\`eme 3.3]{Dru11}. Denote this MMP by $\varphi\colon X\dashrightarrow X_{\min}$ and let $\Gamma:=\varphi_*D_1$. Then
$$K_{X_{\min}}+\delta_1\Gamma=\varphi_*(K_X+\Delta)\sim_\Q \varphi_*D= d_1\Gamma,$$
and in particular, $K_{X_{\min}}+(\delta_1-d_1)\Gamma\sim_\Q 0$. Note that $\Gamma$ is nef and 
$$\kappa(X_{\min},d_1\Gamma)=\kappa(X,K_X+\Delta)=0.$$

Therefore, by replacing $(X,\Delta)$ by $(X_{\min},\delta_1\Gamma)$, we may assume that $\Delta$ is nef, the support of $\Delta$ is either a prime divisor or empty, that $\kappa(X,\Delta)=0$, and there exists a rational number $0<\delta<1$ such that $K_X+\delta\Delta\sim_\Q0$.

\medskip

\emph{Step 5.}
Let $\pi\colon X\dashrightarrow Z$ be the nef reduction of $\Delta$, see \cite{BCE+}. If $\dim Z=0$, then $\Delta\equiv 0$ and we are done. If $\dim Z=\dim X$, then $\Delta\cdot C>0$ for every curve $C$ on $X$ passing through a very general point on $X$, a contradiction by the Nonexistence conjecture.

Therefore, we may assume that $0<\dim Z<\dim X$. As in Step 1 of the proof of Theorem \ref{thm:MRC} we may assume that the pair $(X,\Delta)$ is log smooth and that $\pi$ is a fibration; moreover, by \cite[Lemma 2.3]{LP20a} we have $\nu(F,K_F+\Delta|_F)=0$ for a very general fibre $F$ of $\pi$: this follows since a nef reduction is an almost holomorphic map. Then we conclude as in Step 2 of the proof of Theorem \ref{thm:MRC}.
\end{proof}

In a special case, one can say more. The following result was motivated by a question from F.\ Meng; he also pointed out to me that the case when $\kappa(X,{-}K_X)>0$ below follows from the case when $\kappa(X,\Delta)>0$.

\begin{thm}\label{thm:RCkappa}
Assume the existence of good models for non-uniruled klt pairs with boundaries with rational coefficients in dimension $n-1$.

Let $(X,\Delta)$ be a projective $\Q$-factorial log canonical pair such that $K_X+\Delta$ is pseudoeffective and $X$ is rationally connected. If $\kappa(X,\Delta)>0$ or $\kappa(X,{-}K_X)>0$, then $(X,\Delta)$ has a good model.
\end{thm}

\begin{proof}
By \cite[Theorem 1.3 and Lemmas 2.3 and 2.4]{LM21} we may assume the existence of good models for log canonical pairs in dimensions at most $n-1$. By \cite[Lemma 4.1]{LM21} it suffices to show that 
\begin{equation}\label{eq:11}
\kappa_{\iota}(X, K_X+\Delta)=\nu(X, K_X+\Delta).
\end{equation}

Assume first that $\kappa(X,{-}K_X)>0$. By Theorem \ref{thm:DL}(a) there exists an $\R$-divisor $R\geq0$ such that $K_X+\Delta\sim_\R R$. In particular,
$$\kappa(X,\Delta)\geq\kappa_\iota(X,\Delta-R)=\kappa(X,{-}K_X)>0.$$
Therefore, for the remainder of the proof I assume that $\kappa(X,\Delta)>0$.

\medskip

\emph{Step 1.}
Let $f\colon Y\to X$ be a log resolution of the pair $(X,\Delta)$. We may write
$$K_Y+\Delta_Y\sim_\R f^*(K_X+\Delta)+E,$$
where $\Delta_Y$ and $E$ are effective $\R$-divisors without common components. Let $G$ be the reduced divisor on $Y$ whose support equals the union of all $f$-exceptional prime divisors on $Y$ whose discrepancies are non-negative. Pick $0<\varepsilon\ll1$ such that the pair $(Y,\Delta_Y')$ is log canonical, where $\Delta_Y'=\Delta_Y+\varepsilon G$. Then we have
$$K_Y+\Delta_Y'\sim_\R f^*(K_X+\Delta)+E+\varepsilon G,$$
and it suffices to show that $\kappa_{\iota}(Y, K_Y+\Delta_Y')=\nu(Y, K_Y+\Delta_Y')$. Note that $\Supp f^*\Delta\subseteq\Supp\Delta_Y'$, hence there exists a positive integer $m$ such that $f^*\Delta\leq m\Delta_Y'$. In particular, we have $0<\kappa(X,\Delta)\leq\kappa(Y,\Delta_Y')$.

Therefore, by replacing $(X,\Delta)$ by $(Y,\Delta_Y')$, we may assume that $(X,\Delta)$ is log smooth.

\medskip

\emph{Step 2.}
Assume first that $\lfloor\Delta\rfloor=0$ and that $\Delta$ is a $\Q$-divisor. If $K_X+\tau\Delta$ is not pseudoeffective for any $\tau<1$, then we conclude by Theorem \ref{thm:tau}(b). Otherwise, as in Step 1 of the proof of Theorem \ref{thm:RC} there exists a $\Q$-divisor $D\geq0$ such that $K_X+\Delta\sim_\Q D$ and $\Supp\Delta\subseteq\Supp D$. Thus, $0<\kappa(X,\Delta)\leq\kappa(X,D)$ as in Step 1 above, and we conclude by Theorem \ref{thm:Lai}.

\medskip

\emph{Step 3.}
Now assume only that $\lfloor\Delta\rfloor=0$. The pair $(X,\Delta)$ has a minimal model $(X',\Delta')$ by \cite[Theorem C]{LT22}. By \cite[Lemma 2.8]{LP18a} we have
\begin{equation}\label{eq:10}
\kappa(X',\Delta')\geq\kappa(X,\Delta)>0.
\end{equation}

By Theorem \ref{thm:ShoBir} there exist finitely many $\Q$-divisors $\Delta_i$ on $X'$ and positive real numbers $r_i$ such that each pair $(X',\Delta_i)$ is klt, each $K_{X'}+\Delta_i$ is nef, $\Delta'=\sum r_i\Delta_i$ and $K_{X'}+\Delta'=\sum r_i(K_{X'}+\Delta_i)$. By Theorem \ref{thm:DL}(a) there exist $\Q$-divisors $D_i\geq0$ such that $K_{X'}+\Delta_i\sim_\Q D_i$, so that
\begin{equation}\label{eq:9}
K_{X'}+\Delta'\sim_\R\sum r_iD_i.
\end{equation}

Pick positive rational numbers $s_i$ such that $\sum s_i=1$, and set $\Delta^\circ:=\sum s_i\Delta_i$ and $D^\circ:=\sum s_iD_i$. Then we have $K_{X'}+\Delta^\circ\sim_\Q D^\circ$  and $\Supp D^\circ=\Supp(\sum r_i D_i)$, and therefore \eqref{eq:9} gives
$$\kappa_\iota(X',K_{X'}+\Delta')=\kappa(X',K_{X'}+\Delta^\circ),\quad \nu(X',K_{X'}+\Delta')=\nu(X',K_{X'}+\Delta^\circ).$$
Since $\Supp\Delta'=\Supp\Delta^\circ$, by \eqref{eq:10} we also have $\kappa(X',\Delta^\circ)=\kappa(X',\Delta')>0$, hence $\kappa(X',K_{X'}+\Delta^\circ)=\nu(X',K_{X'}+\Delta^\circ)$ by Step 2. This implies \eqref{eq:11}.

\medskip

\emph{Step 4.}
Finally, assume that $\lfloor\Delta\rfloor\neq0$. If $K_X+\Delta-\tau\lfloor\Delta\rfloor$ is not pseudo\-effective for any $\tau>0$, then we conclude by Theorem \ref{thm:tau}(a). Other\-wi\-se, pick $0<\tau<1$ such that $K_X+\Delta-\tau\lfloor\Delta\rfloor$ is pseudoeffective. Then for each $0<\tau'\leq\tau$ we have $\Supp(\Delta-\tau'\lfloor\Delta\rfloor)=\Supp\Delta$, hence 
$$\kappa(X,\Delta-\tau'\lfloor\Delta\rfloor)=\kappa(X,\Delta)>0.$$
By Step 3, this implies that each pair $(X,\Delta-\tau'\lfloor\Delta\rfloor)$ has a good model. Then we finish as in Step 4 of the proof of Theorem \ref{thm:MRC}.
\end{proof}

	\bibliographystyle{amsalpha}
	\bibliography{biblio}

\end{document}